\documentclass[12pt, a4paper, reqno]{amsart}
\usepackage[margin=3cm]{geometry}

\usepackage{latexsym,amsmath,amsthm,amssymb,mathrsfs,esint,ulem,mathtools}
\usepackage{graphics}
\usepackage{enumerate}
\usepackage{hyperref}
\usepackage{epsfig}

\usepackage{xcolor}
\numberwithin{equation}{section}

\newcommand{\R}{\mathbb{R}}
\newcommand{\N}{\mathbb{N}}

\renewcommand\O{{\Omega}}

\newcommand{\id}{{\rm \bf{id}}}

\newcommand\p{\partial}

\newcommand\ba{\begin{array}}
\newcommand\ea{\end{array}}

\newcommand\iE{\int_{D\setminus E}}

\newcommand\ipD{\int_{\p D}}

\renewcommand\phi{\varphi}

\newcommand\h{{\mathcal H}^{N-1}}
\newcommand\hh{{\mathcal H}^{1}}
\renewcommand\o{\omega}
\newcommand{\weak}{\rightharpoonup}

\newcommand{\weakst}{\stackrel{\ast}{\rightharpoonup}}
\newcommand{\pscal}[2]{\langle #1, #2 \rangle}
\newcommand{\Hspace}[3]{H_{ #1, #2,  #3}}

\theoremstyle{plain}
\begingroup
\newtheorem{theorem}{Theorem}[section]
\newtheorem{lemma}[theorem]{Lemma}
\newtheorem{proposition}[theorem]{Proposition}

\endgroup

\theoremstyle{definition}
\begingroup

\newtheorem{remark}[theorem]{Remark}

\endgroup

\DeclareMathOperator{\dive}{div}

\def \dis {\displaystyle}

\def \hdiv {\hbox{div}\,}

\def \om {\omega}
\def \Om {\Omega}

\def \ph {\varphi}

\def\LN{{\mathcal L}^{N}}
\def\huk{\hat u_k}
\newcommand{\hur}[1]{{\hat u}_{#1r}}
\newcommand{\hut}[1]{{\hat u_{#1\theta}}}

\def \beq {\begin{equation}}
\def \eeq {\end{equation}}
\def \ba {\begin{array}}
\def \ea {\end{array}}

\definecolor{verde}{RGB}{20,150,100}

\newcommand{\EEE}{\color{black}}

\begin{document}

\title[Two-dimensional velocity optimization]{Velocity optimization of self-equilibrated obstacles in a two-dimensional viscous flow}

\author[G.A. Francfort] {Gilles A. Francfort}

\address[G.A. Francfort]{Flatiron Institute, 162 Fifth Avenue, New York, NY10010, USA}
\email[G. A. Francfort]{gfrancfort@flatironinstitute.org}

\author[A. Giacomini] {Alessandro Giacomini}
\address[A. Giacomini]{DICATAM, Sezione di Matematica, Universit\`a degli Studi di Brescia, 
Via Branze 43, 25123,  Brescia, Italy}
\email[A. Giacomini]{alessandro.giacomini@unibs.it}

\author[S. Weady] {Scott Weady}
\address[S. Weady]{Flatiron Institute, 162 Fifth Avenue, New York, NY10010, USA}
\email[S. Weady]{sweady@flatironinstitute.org}

\begin{abstract}{\scriptsize
An obstacle is immersed in an externally driven 2D Stokes or Navier-Stokes fluid.  We study the self-equilibration conditions for that obstacle under steady state assumptions on the flow. We then seek to optimize the translational and/or angular velocity of the obstacle by varying its shape.  To allow general variations, we must consider a very large class of obstacles for which the notion of trace is meaningless. This forces us to revisit the notion of self-equilibration for both Stokes and Navier-Stokes in a measure theoretic environment.

\vskip.5cm

 \noindent  {\it Mathematics subject classification 2020} :  49J45, 49Q10, 49Q20, 74F10, 76D05, 76D07, 76D55,

 \noindent {\it Keywords} : Stokes, Navier-Stokes, mobility, self-equilibration, capacity}

\end{abstract}

\maketitle

\section{Introduction}

In this  contribution, we aim to consider the issue of ``self-equilibrated" obstacles in a viscous flow. Assume that a rigid obstacle  $E$ is immersed in an incompressible viscous fluid flow (with density $\rho$ and viscosity $\mu$) which thus obeys Navier-Stokes equations.  Classically, the obstacle is placed in a linear flow  such that the far field velocity is of the form  $u=A(x-x_*)+a$ where $A\in \R^{N\times N}$ is trace free (to respect incompressibility), $a\in \R^N$ is a translational velocity  and $x_*$ is some arbitrary point in $E$ \cite{HB2012}.  The  rigid obstacle  will be set in motion because of the drag forces exerted by the fluid. In particular  the  point $x_*\in E$ will possess an acceleration $\gamma_*$ and a  velocity $b$; the solid will also possess an inertia matrix $I_*$ with respect to  that point and an angular velocity $\o$ (which, for $N=2$, we will take to be  the $3$-vector $\o=\o \vec{e}_3$), all of those, $E,x_*, \gamma_*, b,I_*,\o$ being a priori time-dependent. Recall that the inertia matrix  about the point $x_*$ is given, for any vector $\alpha\in \R^N$, by
\beq\label{eq.in-def}
\dis I_*\alpha:=\int_E (|x-x_*|^2\alpha-((x-x_*)\cdot\alpha) (x-x_*))\ dx= \int_E (x-x_*)\wedge(\alpha\wedge(x-x_*))\ dx.
\eeq
\par
The resulting set of equations is (more or less) classically seen to  be
$$
\left\{\begin{array}{lr}
\dis\rho\left\{\frac{\p u}{\p t}+ (u\cdot\nabla) u\right\}-\mu\triangle u+\nabla p-\rho \vec{g}=0,& \mbox{ outside }E\\[3mm]
\hdiv u=0,& \mbox{ outside }E\\[3mm]
u\to  A(x-x_*) +a,&\mbox{ as } |x|\to \infty\\[3mm]
u=\o \wedge (x-x_*)+b,&\mbox{ on }\p E
\end{array}\right.
$$
together with
$$
\left\{\begin{array}{l}
\dis  m_E \gamma_*+\frac{d\o}{dt}\wedge\rho_E \left(\int_E (x-x_*)\ dx\right)+\o\wedge\dis\left[\o\wedge\rho_E \left(\int_E (x-x_*)\ dx\right)\right]\\[3mm]\dis\qquad\qquad \qquad\qquad\qquad\qquad \qquad\qquad \qquad\qquad =\int_{\p E}\sigma\nu d\h+ m_E  \vec{g}
\\[3mm]
\dis I_*\frac{d\o}{dt}+\o\wedge (I_*\o)+\rho_E \left(\int_E (x-x_*)\ dx\right)\wedge \gamma_* \\[3mm]\dis\qquad\qquad \qquad\qquad=\int_{\p E}(x-x_*)\wedge\sigma\nu d\h+\rho_E \left(\int_E (x-x_*)\ dx\right)\wedge \vec{g}
\end{array}\right.
$$
where $ m_E :=\rho_E \LN(E)$ is the (time-independent) mass of the obstacle, $\rho_E$ its density and where $\vec{g}:=-g\vec{e}_3$ is gravity.  The point $x_*$ is not necessarily the center of mass,  and this gives rise to the additional acceleration terms on the left hand side.  Of course,  the system would have to be complemented by appropriate initial conditions for  the velocity $u$ of the fluid,  the initial   position of  $x_*$ and an initial velocity for $b$, as well as an initial  angular velocity for $\o$.

\medskip

In the sequel we address two different problems. First, we consider the classical Stokes {\it mobility problem}  for a force-free obstacle in a background flow. This setting commonly arises in the study of particle suspensions, such
as transport in pressure-driven pipe flow, where an external velocity is prescribed and the goal is to determine the particle motion  \cite{HB2012}.  In view of the linear character  of the Stokes equations, this problem is also closely related to the self-propulsion of micro-organisms whose surface velocity is prescribed; see e.g. \cite{Lauga2009, Galdi99}.  Here we assume that the Reynolds number of the fluid  (essentially $\rho/\mu$) is very small so that inertial effects are neglected, and the obstacle is neutrally buoyant, that is its mass density is assumed to be that of the surrounding fluid. In that case, upon taking the point $x_*$ to be at the origin, the system becomes
\beq\label{eq.mob}
\left\{\begin{array}{lr}
-\mu\triangle u+\nabla p=0,& \mbox{ outside }E\\[3mm]
\hdiv u=0,& \mbox{ outside }E\\[3mm]
u\to Ax+a,&\mbox{ as } |x|\to \infty\\[3mm]
u=\o\wedge x+b,&\mbox{ on }\p E,
\end{array}\right.
\eeq
 together with
\beq\label{eq.selfequi}
\left\{\begin{array}{rcl}\int_{\p E}\sigma\nu d\h&=&0\ \\[3mm]\int_{\p E}x\wedge\sigma\nu d\h&=&0.\end{array}\right.
\eeq
The solution to this problem depends only on the domain $E$ and on the far-field flow $Ax + a$. While in principle the domain will change in time, we restrict our attention  to the {\it quasi-static} approximation in which the unknown fields $u$, $\o$, $b$ and the domain $E$ are viewed as constant in time and the goal is to find $\o,b$ such that the system \eqref{eq.mob}, \eqref{eq.selfequi} is solvable in $u$.

Second, we consider the Navier-Stokes mobility problem under the assumption of negligible acceleration of the obstacle, but retain the nonlinearity in the hydrodynamic equations.   Then in particular $\gamma_*=0$ and $d\o/dt=0$. We again  use a  quasi-static approximation for which $u, \omega, b$ and $E$ are viewed as constant in time,  this time  in a reference frame moving with the rigid body velocity $b$.  Doing so  is no capricious whim,   but rather because our proof of the existence of a self-equilibrated motion of the obstacle in Subsection \ref{subsec:ses-navier} collapses in a fixed reference frame.  
In this setting the boundary condition at infinity has become $u\to A(x-x_*)+a-b$ while that on $\p E$ is still $u=\o\wedge (x-x^*)$.

All of this results in the following system (at a fixed time $t$): 
\beq\label{eq.qs-fluid}
\left\{\begin{array}{lr}
\rho (u\cdot\nabla) u-\mu\triangle u+\nabla p- \rho \vec{g}=0,& \mbox{ outside }E\\[3mm]
\hdiv u=0,& \mbox{ outside }E\\[3mm]
u\to  A(x-x_*)+a-b,&\mbox{ as } |x|\to \infty\\[3mm]
u=\o\wedge (x-x_*),&\mbox{ on }\p E
\end{array}\right.
\eeq
with $A,a,\o, b$ constant, together with
\beq\label{eq.qs-solid}
\left\{\begin{array}{lr}
\dis 0=\int_{\p E}\sigma\nu d\h-\o\wedge\dis\left[\o\wedge\rho_E \left(\int_E (x-x_*)\ dx\right)\right]+ m_E \vec{g}
\\[3mm]
\dis 0=\int_{\p E}(x-x_*)\wedge\sigma\nu d\h+(I_*\o)\wedge\o+\rho_E \left(\int_E (x-x_*)\ dx\right)\wedge \vec{g}.
\end{array}\right.
\eeq
Once again the goal is to find $\o,b$ such that the system  \eqref{eq.qs-fluid}, \eqref{eq.qs-solid} is solvable in $u$. Observe that we could equally choose the point $x_*$ to be outside the obstacle, thereby imposing  that the obstacle rotates at constant angular velocity around that point.
\par
In both settings we say that the obstacle is self-equilibrated.

\medskip
 Before we proceed further, we must clarify and qualify the preceding formulation. In  our initial writing of the coupled Navier-Stokes/rigid solid system, we chose an inertial fixed reference frame. Since the Navier-Stokes equations are invariant under a Galilean change of frame, we are free to write that system in a frame that moves at constant velocity, for example at velocity $b$. However, imposing time-independence of the fluid velocity in that new frame is {\it not} equivalent to the same condition in the original fixed reference frame. 

Similarly, we could write the equations in a non-inertial frame in which the solid is at rest, providing yet another time-independence of the fluid velocity.   Doing the latter raises the following  issue: Do the Navier-Stokes equations still hold unchanged in a non-inertial framework at the expense of accommodating the centrifugal and Coriolis accelerations? Our ``wille zur macht" seems to vindicate  the assertion but the resulting system which involves centrifugal and Coriolis terms in the Navier-Stokes equation is more intricate than that in the translating frame while boundary conditions become more challenging. This is why   we confine the non-inertial framework to the     related comment in Remark \ref{rem.wzm}.  \EEE

\medskip

Galdi \cite{Galdi99}  established  that a solution to the (first) mobility problem exists   for smooth enough obstacles, that is that one can find $(\o,b)$ such that the resultant and torque of the forces that the fluid applies to $E$ are $0$. We will revisit this question in Section \ref{sec.galdi} and show that one can generalize Galdi's results to obstacles that are (essentially) merely compact sets  (Theorem \ref{thm.Gg} in any dimension).
\par
In Subsection \ref{subsec:ses-navier}, we will return to the second (quasi-static) problem and show that one can find a solution $u$ to Navier-Stokes and a pair $(\o,b)$ such that, {\it provided that the obstacle $E$ is neutrally buoyant}, the resultant and torque of the forces that the fluid applies to $E$ satisfy \eqref{eq.qs-solid} for obstacles that are once more (essentially) compact sets. 
\par
Remark that the  system  we study in Section \ref{sec.galdiNS} differs from that studied in \cite{Galdi99} in  particular because we do not write the equations in a frame attached to the body.  As already stated, although  doing so has the advantage of leaving the obstacle $E$ fixed, it begs other modeling questions that we do not wish to confront in this work.
\vskip10pt
In both settings, the optimization  problem we wish to address next  is the following: can we find an obstacle $E$ for which the velocity torsor $(\o,b)$ is as small, or as large as possible. In truth, we would like to find a shape for the obstacle for which that torsor is such that $\o=0$. This would imply that, under a steady shear flow, the obstacle only translates. This we cannot do for now, except for a very degenerate 2d mobility problem; see Section \ref{sec.example}.

In Section \ref{sec.min}  we will establish (Theorem \ref{thm:shape} and Remark  \ref{rem:constraint}) that, modulo mild restrictions on the obstacle, we can always find an optimal shape for both the minimization or the maximization problem for that torsor for the   mobility problem. In Subsection \ref{subsec:stab-navier},we will do the same for the quasi-static problem (Theorem \ref{thm:shape-navier}). The methods for achieving those results differ significantly. In the Stokes case
we make full use of the linearity of the problem and of its variational nature, while in the Navier-Stokes case we use a fixed point argument of the kind used for the elementary proofs of existence. Both settings are further complicated by the non-smoothness of the obstacle which forces us to appeal to various measure theoretic arguments. 

Our results will hold true under two important caveats. First, we will confine our investigation to a large box $D$ so that the  far field flow will be imposed on $\p D$ and not at infinity. This is so because we wish to avoid the added complexity of  incompressible inviscid  flows in infinite domains and related issues like the Stokes paradox (see e.g. \cite[Chapter V]{Galdi-book}). We also note that nothing in our analysis is contingent upon the consideration of a  linear flow on $\p D$. Any boundary condition of the form $u=\psi$ with $\psi\in H^1(D;\R^N)$ and ${\rm div\;}\psi=0$ will do.  Consequently, the third relation in \eqref{eq.qs-fluid} and \eqref{eq.mob}, respectively, will be replaced by
\beq\label{eq.gbc}
u= \psi -b \;\mbox{in Section \ref{sec.galdiNS} }\mbox{(resp. $u=\psi$ in Sections \ref{sec.galdi}, \ref{sec.min})}\mbox{ on }\p D, 
\eeq
 with  $\psi\in H^1(D;\R^N){\rm \;and\; div\;}\psi=0$. 
 
 Second and more significantly, while we can revisit Galdi's problem in any dimension (generalizing the notion of torque)  for general geometries of the obstacle,  only in the two-dimensional case can we take advantage of such generality when going to shape optimization  (see the comments at the beginning of Section \ref{sec.min}).  This is so because our analysis relies on a  result of Sverak \cite{Sverak93}, itself based on purely two-dimensional arguments; see our stability result Theorem \ref{thm:stability} and the proof of Theorem \ref{thm:shape-navier}.
 
Our results are  only a first step even in a two dimensional setting. Whether one could achieve $\o=0$ for any flow of the form $u=Ax+a$ on $\p D$ remains unclear although  we give an affirmative answer through the example of a horizontal needle in Section \ref{sec.example} in the case of a pure shear flow, that is when $a=0,\;A=\left(\begin{smallmatrix}0&1\\[1mm]0&0\end{smallmatrix}\right)$. In that section, we also compute the angular velocity for a disk and a vertical needle when subjected to the same pure shear flow. All of those are achieved in the setting of the Stokes mobility problem.

In a subsequent step, one should investigate the topological regularity of the possible minimizers or maximizers which, in view of the result of in Section \ref{sec.example} may prove quite challenging from the standpoint of the classical methods that use the concept of shape derivative.

\medskip

Notationwise, we define, for any $a,b\in \R^d$, the commutator
$[a,b]:= a\otimes b-b \otimes a$,  an element of $\R^{d(d-1)/2}$ . We denote the identity matrix on $\R^d$ by $\id$. The Euclidean scalar product on $\R^d$ and the Frobenius inner product on $\R^{d\times d}$ are both denoted by $\cdot$\;. The basis $\vec{e}_i,i=1,...,N,$ is the canonical orthonormal basis of $\R^N$ with $\vec{e}_3$ pointing upwards if $N=3$ and gravity is involved. By $1_E$ we denote the characteristic function of the set $E$, by $B(x,r)$ the open ball of center $x$ and radius $r$ in the ambient space $\R^N$  and by $1$  the $d$-vector with entries $1$ with $d=1,N$ or $N(N-1)/2$. 

For constants, the notation $C_{F}$ means that the constant, whose value may change from line to line, is only a function of the  quantity $F$.\

By ${\rm cap}_2$ we denote the $2$-capacity, an outer measure on $\R^N$; see e.g. \cite[Section 4.7]{EG}. 

We also introduce the following notation 
\beq\label{eq.duality}\begin{cases}\langle f, u\rangle_{\partial E}:=  \sum_{i=1}^d \langle f_i, u_i\rangle_{H^{-1/2}(\p E),H^{1/2}(\p E)}\vec{e}_i,\\ \\
 \langle f, u\rangle_{\partial D}:= \sum_{i=1}^d\langle f_i, u_i\rangle_{H^{-1/2}(\p D),H^{1/2}(\p D)}\vec{e}_i 
\end{cases}\eeq 
for $ f\in H^{-1/2}(\p E;\R^d), u\in H^{1/2}(\p E;\R^d)$ and $ f\in H^{-1/2}(\p D;\R^d), u\in H^{1/2}(\p D;\R^d)$ respectively.
 In what follows  $d$ will be $1,N$, or $N(N-1)/2$, depending on the context.\EEE

The rest is standard notation.

\section{A generalized setting -- The Stokes case}\label{sec.galdi}
In this section we prove a Galdi type theorem \cite{Galdi99} for compact obstacles $E$  lying inside a large $N$-dimensional open simply connected bounded  box $D$ with sufficiently smooth boundary. The fluid boundary condition on $\p D$ will be given by \eqref{eq.gbc}.

  We first consider in Subsection \ref{sub.Lip} the case where $E$ is the closure of a domain with Lipschitz boundary, so that the Stokes equations can be classically solved and the conditions of self-equilibrium involve integrations on regular boundaries. This  will in turn suggest the proper generalization of the problem in the case of general obstacles which we address in Subsection \ref{sub.Gen}.

\subsection{Obstacles with Lipschitz boundary}\label{sub.Lip}
 Assume that $E\subset D$ compact is the closure of a Lipschitz domain.

\begin{theorem}[\bf Galdi's theorem]\label{thm.G} 
 Let $D\subset \R^N$ be open bounded with Lipschitz boundary and $\psi \in H^1(D;\R^N)$ with ${\rm div\;}\psi=0$. Let $E\subset D$ be the closure of a Lipschitz open set and \EEE
consider $(u_{E,\O,b}, p_{E,\O,b}) \in H^1(D\setminus E;\R^N)\times L^2(D\setminus E)$   the solution to Stokes equation 
\beq\label{eq.stokes}
\begin{cases}
-\mu\triangle u+\nabla p=0\\[3mm]
\dive u=0
\end{cases}\eeq
on $D\setminus E$ with boundary conditions 
\beq\label{eq.abc}
\begin{cases}
u= \psi  &\mbox{ on }\p D\\[3mm]
u=\Omega x+b &\mbox{ on }\p E,
\end{cases}\eeq
 where $(\Omega,b)\in \R^{N\times N}\times \R$ with $\Om$ skew-symmetric.
\par
Then there exists a  unique  pair $(\Omega_E, b_E)$ such that, setting 
$$
(u_E,p_E):=(u_{E, \O_E, b_E}, p_{E, \O_E, b_E}),
$$ 
the associated stress $\sigma_{E}:= 2\mu e(u_{E})-p_{E}\;\! \id$  
satisfies
\beq\label{eq.selfeq}
\langle\sigma_{E}\nu,1\rangle_{ \p E}=0\qquad\text{and}\qquad
\langle[\sigma_{E}\nu, x], 1\rangle_{ \p E}=0
\eeq
 where $\nu$ is the exterior normal to $\p E$ and $\langle\cdot,\cdot\rangle_{\p E}$ is defined in \eqref{eq.duality}.\EEE
\end{theorem}

\begin{proof}
 We divide the proof into three steps.

\vskip10pt\noindent{\bf Step 1: An auxiliary problem.} In this step, we establish that, if  $\psi=0$ and if \eqref{eq.abc}, \eqref{eq.selfeq} holds true, then 
\beq\label{eq.zero}
\O=0 \qquad\text{and}\qquad b=0.
\eeq

Indeed, testing equation \eqref{eq.stokes} by $u:=u_{E,\O,b}$  (with associated stress $\sigma:=\sigma_{E,\O,b}$) we obtain
\begin{multline*}
0= 2 \mu \iE |e(u)|^2\ dx+ \langle \sigma\nu, (\Omega x+b)\rangle_{ \p E}\\= 2 \mu \iE |e(u)|^2\ dx+ \langle [\sigma\nu,x],1 \rangle_{ \p E}\cdot\O+ \langle \sigma\nu,1\rangle_{ \p E}\cdot b= 2 \mu \iE |e(u)|^2\ dx,
\end{multline*}
which implies that $e(u)\equiv 0$ in the open connected component $\hat D$  (with Lipschitz boundary) of $D\setminus E$ whose boundary contains $\p D$.  On that component $u$ must be  an infinitesimal  rigid body motion of the form $\bar\O x+\bar b$, with
$\bar \O$ skew-symmetric. But $u=0$ on $\p D$ 
 so that $\bar \O=0$ and $\bar b=0$.  Since $u=\O x+b$ on $\partial \hat D\cap \partial E$, we conclude that  $\O=0$ and $b=0$.

\medskip

\noindent{\bf Step 2:  A family of  elementary problems.} Take now  $\psi=0$ and $\O^{ij}_{kl}=-\O^{ij}_{lk}=\delta_{ik}\delta_{jl}, i\ne j, \; b=0$, or else $b^i=\vec{e}_i, \O=0$ with  $i,j=1,\dots, N$ . As $i,j$ vary from $1$ to $N$, this defines $ {N(N+1)}/{2}$ Stokes problems with solution pairs $(u,p)$ and their associated stress $\sigma=2\mu e(u)-p\;\!\id$.
We claim that the associated $ {N(N+1)}/{2}$ pairs $(\langle \sigma\nu,1\rangle_{\p E }, \langle [\sigma\nu,x],1\rangle)_{ \p E}$ are independent vectors in $\R^{ {N(N+1)}/{2}}$ and thus form a basis of $\R^{ {N(N+1)}/{2}}$ which we will denote by $\{F_k\}$.

Indeed, otherwise a non trivial linear combinations of those would be nul. But, by linearity, this would generate a non-zero pair $(\O,b)\in \R^{N\times N} \times \R^N$ with $\O$ skew-symmetric such that the  auxiliary problem of Step 1 is satisfied which contradicts \eqref{eq.zero}.

\medskip

\noindent{\bf Step 3.}  Solve Stokes equation on $D\setminus E$ under the boundary conditions \eqref{eq.abc} and $\Om=0$ and $b=0$ on $\p E$.  This will yield a solution  pair $(u,p)$ with associated stress $\sigma$ generating a $ {N(N+1)}/{2}$-vector $F:=(\langle\sigma\nu,1\rangle_{ \p E},\langle [\sigma\nu,x],1\rangle_{ \p E})$. 
\par
 In view of  Step 2, there exists a  unique ${N(N+1)}/{2}$-vector  $(\alpha_1, ..., \alpha_{  {N(N+1)}/{2}})$
 such that 
 \beq\label{eq.int1}
 F=\sum_{k=1}^{ {N(N+1)}/{2}} \alpha_k F_k,
 \eeq  
  where $\{F_k\}$ is the basis defined in Step 2.
 To that linear combination, we associate the corresponding linear combination of the $b^i$'s and $\O^{ij}$'s which yields a pair $\bar b,\bar \O$ such that the solution $(\bar u,\bar p)$ to Stokes with boundary conditions $\bar u=0$ on $\p D$ and $\bar u= \bar \O x+\bar b$ on $\p E$ yields a $ {N(N+1)}/{2}$-vector $(\langle\bar \sigma\nu,1\rangle_{ \p E},\langle [\bar \sigma\nu,x],1\rangle_{ \p E})$ equal to $F$.
 
 But then $(\check u:= u-\bar u, \check p:=p-\bar p)$ is the solution to Stokes equation \eqref{eq.stokes} with boundary conditions $\check u=\psi$ on $\p D$ and $\check u=-\bar \O x-\bar b$ on $\p E$ and, in view of  \eqref{eq.int1}, it is self-equilibrated.
 \end{proof}
 
 \begin{remark}\label{rk.switch}
   By the divergence theorem, \eqref{eq.selfeq} equivalently reads as
$$
 \begin{cases}
 \langle\sigma_E\nu,1\rangle_{ \p D}=0,\\[3mm]
\langle[\sigma_E\nu, x], 1\rangle_{ \p D}=0,
\end{cases}
 $$
 which has the distinct advantage of integrating over the boundary of $D$ in lieu of that of the obstacle $\p E$.
 \hfill\P
 \end{remark}

 \subsection{The general case}\label{sub.Gen}
 With the help of   Remark \ref{rk.switch}, we are at liberty to do away with the Lipschitz regularity of $\p E$ and to only assume that

\beq\label{eq.domain}
E\subset D  \mbox{ is a compact  set with } {\rm cap_2}(E)>0.
\eeq
 The resulting open fluid  domain $D\setminus E$ may thus have a very ``rough"  or singular  boundary near $E$. 
 \par
To lend  meaning to the Stokes problem \eqref{eq.stokes} with boundary conditions \eqref{eq.abc},
we introduce the following  subset of $H^1(D)$
\begin{multline}
\label{eq:Hob}
 \Hspace{E}{\O}{b}(D):=\text{Closure}_{H^1(D;\R^N)}\{\phi \in C^\infty(\bar D;\R^N): \dive\phi=0,\\
 \phi \equiv (\O x+b)  \mbox{ in a neighborhood of }E\}
\end{multline}
 
and interpret the problem as
\begin{equation}
\label{eq:Stokes-gen}
\begin{cases}
(u,p)\in \Hspace{E}{\O}{b}(D)\times L^2_{loc}(D\setminus E) &\\[3mm]
-\mu\triangle u+\nabla p=0 &\text{in }D\setminus E\\[3mm]
u=\psi &\text{on }\partial D.
\end{cases}
\end{equation}

 A solution $(u_{E,\O,b}, p_{E,\O,b})$ is given after the following remark.

 \begin{remark}\label{rem.equiv-sp} An equivalent way of defining the space $\Hspace{E}{\O}{b}(D)$ consists in imposing
 the condition $\phi = (\O x+b)$ solely on the obstacle $E$, that is defining $\Hspace{E}{\O}{b}(D)$ as
 $$
 \Hspace{E}{\O}{b}(D):= \text{Closure}_{H^1(D;\R^N)}\{\phi \in C^\infty(\bar D;\R^N): \dive\phi=0,\\
 \phi \equiv (\O x+b)  \mbox{ on }E\}.
 $$
 This highly non trivial result is a direct consequence of the following. In view of  \cite[Theorem 5.11]{KLV}, 
 if $u$ is the $H^1$-limit of a sequence $\ph_n\in C^\infty(\bar D;\R^N)$
 with $\ph_n= \O x+b$ on $E$ then $u-(\O x+b)=0 \;{\rm cap_2 \; a.e.\; on \;}E$. But then, according to 
\cite[Theorem 9.1.3]{AH}, $\ph (u-(\O x+b))\in H^1  _0(D\setminus E;\R^N)$ for any 
$\ph \in C^\infty_0(D;\R^N)$ which is trivially equivalent to our definition \eqref{eq:Hob} of the space
$\Hspace{E}{\O}{b}(D)$.
 \hfill\P\end{remark}

 \medskip

 First, the fluid velocity field  is the unique minimizer $u_{E,\O,b}$ of 
\beq\label{eq.minsbc}
 \min\left\{\int_{D\setminus E}|\nabla u|^2\ dx\,:\, u\in  \Hspace{E}{\O}{b}(D): u= \psi  \mbox{ on }\p D\right\}.
\eeq
 From \eqref{eq.minsbc}  the Stokes equation is recovered. Indeed,  the Euler-Lagrange equation immediately implies that $\Delta u_{E,\Om,b}\in H^{-1}(D\setminus E;\R^N)$ vanishes on smooth divergence free vector fields with compact support in $D\setminus E$. By De Rham's theorem which holds true within a   distributional framework (see e.g. \cite{W}),  there exists a distribution $p_{E,\Om,b}\in {\mathcal D}'(D\setminus E)$ such that  
$$
\mu \Delta u_{E,\Om,b}=\nabla p_{E,\Om,b}\qquad\text{in }D\setminus E.
$$
In view of Lemma \ref{lem:lions}  below,  
$p_{E,\Om,b}\in L^2_{loc}(D\setminus E)$, with $p_{E,\Om,b}$ square summable near $\partial D$. If $D'$ is a Lipschitz open set such that $E\subset\subset D'\subset\subset D$, by adding a suitable constant to $p_{E,\Om,b}$ on the connected component of $D\setminus E$ which contains $\partial D$ , we can also  assume, thanks to \eqref{eq:lions} in Lemma \ref{lem:lions}, that 
\begin{equation}
\label{eq:est-p}
\|p_{E,\O,b}\|_{L^2(D\setminus \overline D')}\le C_{D\setminus \overline D'} \|u_{E,\Om,b}\|_{H^1(D;\R^N)}.
\end{equation}
Note that  the divergence-free stress $\sigma_{E,\Om,b}:= 2 \mu e(u_{E,\Om,b})-p_{E,\Om,b} \id$ is such that the quantities  
$\langle\sigma_{E,\Om,b} \nu,1\rangle_{ \p D}$ and $\langle[\sigma_{E,\Om,b} \nu, x], 1\rangle_{ \p D}$
defined through the first equation in \eqref{eq.duality} are meaningful.
\par
 Lemma \ref{lem:lions} is a well-known result but we could not locate a proof in its full generality, which is why we provide a proof below. 

\begin{lemma}
\label{lem:lions}
Let $U \subset \R^N$ be an open bounded set with Lipschitz boundary. Let $p\in {\mathcal D}'(U)$ be such that $\nabla p\in H^{-1}(U;\R^N)$. Then $p\in L^2(U)$ and there exists a constant $C_U$ such that
$$
\|p\|_{L^2(U)}\le C_U\left[ \|p\|_{H^{-1}(U)}+\|\nabla p\|_{H^{-1}(U;\R^N)}\right].
$$
 and
 \begin{equation}
\label{eq:lions}
\left\|p-\fint_U p\,dx\right\|_{L^2(U)}\le C_U\|\nabla p\|_{H^{-1}(U;\R^N)}.
\end{equation}
\end{lemma}

\begin{proof}
Fix $\psi\in C^\infty_c(U)$ with $\int_U \psi\,dx=1$. Then, for every $\varphi\in C^\infty_c(U)$,  Bogovski's theorem  (see \cite[Theorem 1]{Bogovski79} and also \cite[Lemma III.3.1 and Theorem III.3.1]{Galdi-book}) implies the existence of $F_\varphi \in C^\infty_c(U;\R^N)$ such that $\dive F_\varphi=\varphi-\psi\int_U \varphi\,dx$ and
$$
\|F_\varphi\|_{H^1_0(U;\R^N)}\le  C_U\left\| \varphi-\psi \int_U \varphi\,dx\, \right\|_{L^2(U)} \le C_U(1+\| \psi\|_{L^2(U)} (\LN(U))^{1/2})  \| \varphi\|_{L^2(U)}.
$$
Then,
\begin{multline}
\label{eq:ineq-lions}
\pscal{p}{\varphi}=\pscal{p}{ \varphi-\psi \int_U \varphi\,dx}+\pscal{p}{\psi}\int_U \varphi\,dx\\
=\pscal{p}{\dive F_\varphi}+\pscal{p}{\psi}\int_U \varphi\,dx=-\pscal{\nabla p}{F_\varphi}+\pscal{p}{\psi}\int_U \varphi\,dx\\
\le \|\nabla p\|_{H^{-1}(U;\R^N)}\|F_\varphi\|_{H^1_0(U;\R^N)}+
|\pscal{p}{\psi}|\int_U |\varphi|\,dx\\
\le 
 [C_U(1+\| \psi\|_{L^2(U)} (\LN(U))^{1/2})\|\nabla p\|_{H^{-1}(U;\R^N)}+|\pscal{p}{\psi}| (\LN(U))^{1/2}] \|\varphi\|_{L^2(U)}.
\end{multline}
The previous inequality entails that $p$ can be extended to an element in the dual of $L^2(U)$. By Riesz theorem,   $p\in L^2(U)$.
\par
To recover the first estimate of the theorem, it is sufficient to note that $p\in L^2(U)$ yields in particular $p\in H^{-1}(U)$ so that we can estimate  the quantity $|\pscal{p}{\psi}|$ in \eqref{eq:ineq-lions} with $\|p\|_{H^{-1}(U)}\|\psi\|_{H^1_0(U)}$, hence  the result. 

\par
To recover \eqref{eq:lions}, we can assume that $p$ has zero mean and use again \eqref{eq:ineq-lions} 
to get
\begin{multline*}
\int_U p\varphi\,dx\le C_U \|\nabla p\|_{H^{-1}(U;\R^N)} \left\| \varphi-\psi\int_U \varphi\,dx\right\|_{L^2(U)}+\left|\int_U p\psi\,dx\right|
\int_U |\varphi|\,dx.
\end{multline*}
Letting $\psi$ converge to $\frac{1}{\mathcal L^N(U)}1_U$, 
$$
\int_U p\varphi\,dx\le C_U  \|\nabla p\|_{H^{-1}(U;\R^N)} \left\| \varphi-\fint_U \varphi\,dx\right\|_{L^2(U)}
$$
so that \eqref{eq:lions} follows by letting $\varphi$ converge to $p$ in $L^2(U)$.
\par
The result follows by choosing $C_U$ to be $2C_U$. 
\end{proof}

\begin{remark}
\label{rem:bound-uE}
Let $D'$ be a Lipschitz open set such that $E\subset\subset D'\subset\subset D$. 
There exists a constant $C_{D,D'}>0$  such that
\begin{equation}
\label{eq:bound}
\|u_{E,\Om,b}\|_{H^1(D)}\le C_{D,D'}\left( \|\psi\|_{H^1(D;\R^N)}+|\O|+|b|\right).
\end{equation}
Indeed consider a cut-off function $\phi\in C^\infty(\R^N)$ with $\phi\equiv 1$ in a neighborhood of $\R^N\setminus D$ and $\phi\equiv 0$ in a neighborhood of $\bar D'$ and the problem
 $$
 \begin{cases}
 \hdiv w=-\hdiv\left(\phi \psi+(1-\phi)(\Om x+b )\right)&\mbox{ in }D\setminus\bar D'\\[2mm]
 w=0\mbox{ on } \p D\cup\p D'.
 \end{cases}
$$
According to Bogovski's theorem  (see \cite[Theorem 1]{Bogovski79}), there is a solution $w\in H^1_0(D\setminus\bar D';\R^N)$ since
\begin{multline*}
\int_{D\setminus\bar D'}\hdiv\left(\phi \psi+(1-\phi)(\Om x+b )\right)\ dx\\[2mm]=\ipD \psi \cdot\nu\ d\h-\int_{\p D'}(\Om x+b )\cdot \nu\ d\h\\[2mm]=
\int_D\hdiv \psi\ dx-\int_{D'}\hdiv(\Om x+b)\ dx=0.
\end{multline*}
Moreover
\begin{multline*}
\|w\|_{H^1_0(D\setminus \overline D')}\le C_{D,D'}\|\hdiv\left(\phi\psi+(1-\phi)(\Om x+b )\right)\|_{L^2(D\setminus \overline D')}\\
\le  \tilde C_{D,D'}\left( \|\psi\|_{H^1(D;\R^N)}+|\O|+|b|\right).
\end{multline*}
We deduce that
$$
v:=\phi\psi+(1-\phi)(\Om x+b)+w \in \Hspace{E}{\O}{b}(D)
$$
so that the inequality
$$
\int_{D\setminus E}|\nabla u_{E,\Om.b}|^2\,dx\le \int_{D\setminus E}|\nabla v|^2\,dx
$$
entails \eqref{eq:bound}.
\hfill\P\end{remark}

\begin{remark}\label{rmk.erase}
   If $C$ is any  open component of $D\setminus  E$, except that  whose boundary contains $\partial D$, then the solution $(u_{E,\O,b}, p_{E,\O,b})$ to Stokes problem on $C$ is simply $(\O x+b,0)$. 
In other words, we can erase the ``bounded" connected components of   $D\setminus  E$  without changing the problem and could thus assume throughout the rest of the paper that $E$ and its complement are connected  (hence that $E$ is simply connected in dimension 2). We will refrain from doing so because it does not impact any of the subsequent arguments.\hfill\P\end{remark}

\begin{remark}
\label{rem:c2}
Note that, when ${\rm cap_2}(E)=0$,  $\Hspace{E}{\O}{b}(D)=H^1(D;\R^N)$. The integration in \eqref{eq.minsbc} is then  on the entirety of $D$, so that the $E$, $\Om$ and $b$  have no influence on the associated Stokes problem. \hfill\P
\end{remark}

 We claim that Galdi's theorem still applies to this generalized setting.  We need the following result.
 
 \begin{proposition}
\label{prop:galdi0}
Assume \eqref{eq.domain} and let $(\Om,b)\in \R^{N\times N}\times\R^N$ with $\O$ skew-symmetric. Assume that the normal stress $\sigma \nu\in H^{-1/2}(\p D;\R^N)$ on $\partial D$ associated to the solution  $(u,p)$ of \eqref{eq:Stokes-gen} with $\psi=0$ is such that
$$
\langle\sigma\nu,1\rangle_{\partial D}=0\qquad\text{and}\qquad
\langle[\sigma\nu, x], 1\rangle_{ \p D}=0.
$$
Then $(\O,b)=(0,0)$.
\end{proposition}

\begin{proof}
Testing the Stokes equation with $u$ we obtain as before that $e(u)\equiv 0$ in $D\setminus E$. Since $u\in \Hspace{E}{\O}{b}(D)$,  we also obtain by approximation that, as a distribution, $e(u)$ is supported in $D\setminus E$. As a consequence  $e(u)=0$ on the whole of  $D$, which together with $u=0$ on $\partial D$ entails that $u=0$ on $D$. 
\par
 By Remark \ref{rem.equiv-sp}, if $u\in \Hspace{E}{\O}{b}(D)$ and $\tilde u$ denotes its precise representative, then
\begin{equation}
\label{eq:precise}
\tilde u(x)=\O x+b\qquad\text{ for ${\rm cap}_2$-a.e. $x\in E$}.
\end{equation}
Then, since
$$
0=\tilde u(x)=\O x+b \qquad\text{for ${\rm cap}_2$-q.e. $x\in E$},
$$ 
and because the kernel of a non-zero skew-symmetric matrix is at most $(N-2)$-dimensional, we deduce that, unless $\O=0$ and $b=0$, $E$ must be of the form 
$$
E=E_0\cup E_1
$$
with ${\rm cap}_2(E_0)=0$ and $E_1$  contained in a $(N-2)$-dimensional hyperplane so that $\mathcal H^{N-2}(E_1)<\infty$. If $N\ge 3$, by e.g. \cite[Theorem 3, Section 4.7.2]{EG}, ${\rm cap}_2(E_1)=0$, contradicting \eqref{eq.domain}. If $N=2$, then $E_1$, having finite $\mathcal H^0$-measure, is a  finite number of points, hence, once again, ${\rm cap}_2(E_1)=0$, contradicting \eqref{eq.domain}.  We conclude that $\O=0$ and $b=0$.  The result  is proved. 
\end{proof}
\EEE

\begin{remark}
\label{rem:geom}
The proof of Proposition \ref{prop:galdi0} above reduces to establishing that relations $e(u)=0$ in $D\setminus E$ and $u=0$ on $\partial D$ for $u\in \Hspace{E}{\O}{b}$ yield $\O=b=0$.
\par
In the case where  $E$ is regular  enough (see Step 1 in the proof of Theorem \ref{thm.G}) we used the information of the trace of $u$ on $\partial E$ -- a well defined quantity equal to $\O x+b$ -- as a  transmission condition which forces $\O$ and $b$ to vanish. This argument  has apparently disappeared  in the proof Proposition \ref{prop:galdi0}.
Through measure-theoretical geometric arguments, we can also  evidence in our general setting a transmission condition between the irregular boundary $\partial E$ and the region $D\setminus E$. 
\par
More precisely, setting
$$
\p^+ E:=\left\{x\in \p E: \limsup_{r \to 0^+} \frac{\LN(B(x,r)\cap (D\setminus E))}{r^N}>0\right\}
$$ 
we claim that
\beq\label{eq.capd+E}
{\rm cap}_2(\p^+ E)>0.
\eeq
Now, if $u\in \Hspace{E}{\O}{b}(D)$ with $u=0$ on $D\setminus E$, by \eqref{eq:precise} and the continuity of $y\to \O y+b$, for ${\rm cap}_2$-a.e. $x\in \p^+ E$, 
\begin{multline*}
0=\lim_{r\to 0^+}\frac{1}{r^N}\int_{B(x,r)}|u(y)-(\O y+b)| dy\ge \limsup_{r\to 0^+}\frac{1}{r^N}\int_{B(x,r)\cap (D\setminus E)}|\O y+b| dy\\
\ge |\O x+b|\limsup_{r\to 0^+}\frac1{r^N}\LN(B(x,r)\cap (D\setminus E))
\end{multline*}
which yields
$$
\O x+b\equiv 0\mbox{ on }\p^+ E.
$$
But in view of \eqref{eq.capd+E}, we get $\O=0$ and $b=0$ by using the same arguments as in the final part of the proof of Proposition \ref{prop:galdi0} with $\p^+E$ in lieu of $E$.
\par
Let us prove claim \eqref{eq.capd+E}. If false, by e.g. \cite[Theorem 4, Section 4.7.2]{EG}, 
$$\h(\p^+ E)=0.$$ Since the essential boundary $\p_{\rm ess} E$ of $E$ is contained in $\p^+E$,  we deduce that  $\h(\p_{\rm ess} E)=0$. By Federer's theorem \cite[Theorem 1, Section 5.11]{EG}, $E$ is a set of finite perimeter with $1_E=0$, hence a set of $0$-volume, that is such that $\p E=E$. We infer that all points in $E$ have density $1$  for $D\setminus E$. As a consequence $\p^+E=\p E$, so that  ${\rm cap}_2(E)={\rm cap}_2(\p E)=0$ against assumption \ref{eq.domain}.
\hfill\P\end{remark}

We are now in a position to generalize Galdi's result to our setting.

\begin{theorem}[\bf A generalized Galdi theorem] 
\label{thm.Gg} 
Assume \eqref{eq.domain}.Then
there exists a  unique  pair $(\O_E, b_E)\in \R^{N\times N}\times\R^N$ with $\O_E$ skew symmetric
such that, denoting by
$$
(u_E,p_E):=(u_{E, \O_E, b_E}, p_{E, \O_E, b_E})
$$ 
the solution to  \eqref{eq:Stokes-gen}, where  $\psi\in H^1(D;\R^N)$ with ${\rm div\;} \psi=0$, 
the associated stress $\sigma_{E}:= 2 \mu e(u_{E})-p_{E}\;\! \id$
satisfies 
 $$
\langle\sigma_{E}\nu,1\rangle_{ \p D}=0\qquad\text{and}\qquad
\langle[\sigma_{E}\nu, x], 1\rangle_{ \p D}=0,
$$
where $\nu(x)$ is the outer normal to $D$ at $x\in\p D$. 
\end{theorem}

\begin{proof}
The proof is identical to that of Theorem \ref{thm.G} once Step 1 has been replaced by Proposition \ref{prop:galdi0}.
\end{proof}

Note that when $E$ has the regularity required for the application of Theorem \ref{thm.G}, we recover the same solution.

\section{Optimization problems in two dimensions for Stokes flows}
\label{sec.min}
  
We saw in Section \ref{sec.galdi} (Theorem \ref{thm.Gg}) above that  assuming the minimal  condition \eqref{eq.domain} was sufficient to assert the existence of a self-equilibrated velocity field of the form $\O_E x+b_E$ for the obstacle $E$ when the surrounding Stokes flow is  driven by $\psi\in H^1(D;\R^N)$ with ${\rm div\;} \psi=0$ on $\p D$.
\par
We  now address the optimization with respect to the obstacle $E$ \EEE of a functional of the type $E\mapsto f(\O_E ,b_E)$.
 \vskip10pt
Shape optimization problems associated to elliptic equations under Dirichlet boundary conditions are generically ill-posed. This leads to  homogenization effects for example. 
Well-posedness is connected to the issue of {\it stability of solutions} to elliptic problems under the variation of the domain. Such stability statements  hold true  under additional constraints which severely limit the possible geometries: one such setting is  the classical constraint of  equi-Lipschitzianity (see \cite{Ch75} and \cite[Section 4.6]{BB}).
\par
The situation is different in two spatial dimensions thanks to  \cite{Sverak93}: through the use of capacitary arguments, stability is proved there within the class of domains whose complement has a  uniformly bounded number of connected components, thereby allowing a large class of possibly singular variations.   Well-posedness of shape optimization problems, like that of the  minimal drag force in a stationary Stokes/Navier-Stokes setting, is then achieved (see \cite[Section 4]{Sverak93}). Suitable extensions involving  domains  with capacitary constraints have been proposed in \cite{BZ1,BZ2}, and used in \cite{BG} for various shape optimization problems in fluids.
 \vskip10pt
In this section and with \cite{Sverak93} in mind we restrict our focus to the two-dimensional setting. The pair $(\O,b)$ with  the skew symmetric matrix 
 $
 \O=\begin{psmallmatrix}0&-\o\\[2mm]\o&0\end{psmallmatrix}
 $ 
 may be replaced by a pair  $(\o, b) \in \R\times \R^2$  with $\o\vec{e}_3$ as out of plane angular velocity. We will correspondingly define $\o_E$ in lieu of $\O_E$ and use interchangeably $\O$ or $\o$,depending on the context.
 \par

We  first prove (Theorem \ref{thm:stability}) the stability of self-equilibrated flows with respect to the variation of the obstacle $E$ under  Hausdorff convergence (see  \cite[Definition 4.4.9]{AT}).

We then consider  (Theorem \ref{thm:shape})  the optimization (minimization or maximization) with respect to $E$ of the functional
$$
E\mapsto f(\o_E,b_E)
$$
under suitable conditions on the function $f:\R\times \R^2\to \overline\R$ and additional constraints for $E$. This setting will include the minimization or maximization of the modulus of the angular velocity of the obstacle.

\subsection{Stability of self-equilibrated solutions under obstacle perturbation} 
\label{subsec:stability}
 The following theorem holds true.
 
 \begin{theorem}[\bf Stability]
\label{thm:stability}
Let $D\subseteq \R^2$ be open bounded with Lipschitz boundary, and let  $\psi\in H^1(D;\R^2)$ with ${\rm div\;} \psi=0$  be given. Let $E_n$ and $E$ be compact sets in $D$ satisfying \eqref{eq.domain} with a number of connected components uniformly bounded in $n$ and such that
\begin{equation}
\label{eq:Hauss}
E_n\to E\qquad\text{in the Hausdorff topology}.
\end{equation}
Then the self-equilibrated solutions $u_{E_n}, u_E$ given by Theorem \ref{thm.Gg} satisfy
$$
u_{E_n} \weak u_{E}\qquad\text{weakly in }H^1(D;\R^2)
$$
and
$$
\o_{E_n} \to \o_E\qquad\text{and} \qquad b_{E_n} \to b_E.
$$
\end{theorem}

\begin{proof}
For notational simplicity, we set $\o_{E_n}:=\o_n, b_{E_n}:=b_n$ and $u_{E_n} :=u_n$.
We divide the proof into three steps.

\vskip10pt\noindent{\bf Step 1.}
Assume that
$$
\o_n\to \om\qquad\text{and}\qquad b_n\to b
$$
for some $\o\in \R$ and $b\in\R^2$.  Thanks to \eqref{eq:bound} $(u_n)_{n\in\N}$ is bounded in $H^1(D)$, so that, up to a subsequence
\begin{equation}
\label{eq:unweak}
u_n\weak u\qquad\text{weakly in }H^1(D;\R^2).
\end{equation}
Let $\varphi\in C^\infty_c(D\setminus E;\R^2)$ with ${\rm div}\varphi=0$: in view of the Hausdorff convergence \eqref{eq:Hauss}, $\varphi$ is also in $C^\infty_c(D\setminus E_n;\R^2)$ for $n$ large, so that passing to the limit in the Stokes equations on $D\setminus E_n$ we get, in view of \eqref{eq:unweak},
\begin{equation}
\label{eq:weakeq}
\int_{D\setminus E}\nabla u \cdot \nabla \varphi\,dx=0.
\end{equation}
We claim that
\begin{equation}
\label{eq:claim-u}
u\in \Hspace{E}{\O}{b}(D)
\end{equation}
(see \eqref{eq:Hob}).

Indeed, for every open set $A\subseteq D$,  set
$$
H(A):=Closure_{H^1}\{\varphi \in C^\infty_c(A;\R^2)\,:\, \hdiv\varphi=0\},
$$
and let ${\mathcal P}_A$ denote the orthogonal projection of $H(D)$ onto $H(A)$ with respect to the scalar product 
$(\nabla \cdot,\nabla \cdot)_{L^2}$.
\par
Since the number of connected components of $E_n$ is uniformly bounded in $n$,  \cite[Theorem 4.1]{Sverak93} implies that
\begin{equation}
\label{eq:proj}
{\mathcal P}_{D\setminus E_n}\to {\mathcal P}_{D\setminus E}\qquad\text{strongly in }{\mathcal L}(H(D);H(D)).
\end{equation}
Let $D'$ be a Lipschitz domain such that $E\subset\subset D'\subset\subset D$, so that, for $n$ large enough,  $E_n\subset\subset D'\subset\subset D$. Thanks to Bogovski's theorem (see \cite[Theorem 1]{Bogovski79}), we can construct
$\tilde \psi_n, \tilde\psi\in H^1(D;\R^2)$ such that 
$$\begin{cases}
{\rm div\,}\tilde\psi_n={\rm div\,}\tilde\psi=0 \quad\text{in } D,& \tilde \psi_n=\tilde \psi=\psi\qquad\text{on }\p D\\[3mm]
\tilde \psi_n=\o_n \vec e_3\wedge x+b_n\quad\text{in }D',& 
\tilde \psi=\o \vec e_3\wedge x+b\quad\text{in }D'\end{cases}
$$
and
\begin{equation}
\label{eq:psin}
\tilde \psi_n\to \tilde \psi\quad\text{strongly in }H^1(D;\R^2).
\end{equation}
Clearly, by construction
$$
u_n-\tilde \psi_n\in H(D\setminus E_n)
$$
so that passing to the limit in
$$
u_n-\tilde \psi_n={\mathcal P}_{D\setminus E_n}(u_n-\tilde \psi_n),
$$
in view of \eqref{eq:proj}, \eqref{eq:unweak} and \eqref{eq:psin} we deduce that
$$
u-\tilde \psi={\mathcal P}_{D\setminus E}(u-\tilde \psi)
$$
 which yields \eqref{eq:claim-u}.
\par
From \eqref{eq:weakeq} and \eqref{eq:claim-u} we infer that $u=u_{E,\o,b}$. In view of  \eqref{eq:est-p} we can assume that
the pressure $p_n$ associated to $u_n$ is bounded in $L^2(D\setminus \bar D')$, with
$$
p_n\weak p\qquad \text{weakly in }L^2(D\setminus \bar D')
$$
where $p\in L^2_{loc}(D\setminus E)$ is the pressure associated to $u_{E,\o,b}$.
Passing to the associated stresses, this entails
$$
\sigma_{n}\nu \weakst\sigma \nu\qquad\text{weakly* in }H^{-1/2}(\partial D,\R^2),
$$ 
and consequently, because the $u_n$'s are self-equilibrated we obtain
$$
\langle\sigma \nu,1\rangle_{\p D}=\langle[\sigma\nu, x], 1\rangle_{D}=0.
$$
We conclude that $u=u_E$, so that, in particular,  the entire sequence  is such that 
\begin{equation}
\label{eq:unuEbis}
u_n\weak u_E\qquad\text{weakly in }H^1(D;\R^2)
\end{equation}
while
\begin{equation}
\label{eq:onbnE}
\o=\o_E\qquad\text{and}\qquad b=b_E.
\end{equation}

\vskip10pt\noindent{\bf Step 2.}
Let us show that
$$
|\o_n|+|b_n|\le C
$$
for some $C>0$. By contradiction, assume that $\alpha_n:=|\o_n|+|b_n|\to +\infty$. Up to a subsequence we can assume that
$$
\frac{1}{\alpha_n}\o_n\to \o\qquad\text{and}\qquad \frac{1}{\alpha_n}b_n\to b
$$
for some $\o\in \R$ and $b\in\R^2$ such that $(\o,b)\not=(0,0)$. Notice that $(\frac{1}{\alpha_n}u_n,\frac{1}{\alpha_n}p_n)$ is the self-equilibrated solution of the Stokes problem for $E_n$ under the boundary displacement 
$$
\frac{1}{\alpha_n}\psi\to 0
$$
Following the arguments of Step 1 -- the only difference being that the boundary velocity is not fixed but converges to $0$ strongly --  yields 
$$
\frac{1}{\alpha_n}u_n\weak u^0_E\qquad\text{weakly in }H^1(D;\R^2)
$$
where $u^0_E$ is the self-equilibrated solution for $E$ relative to the zero boundary condition on $\partial D$, and with associated non trivial rigid body motion on $E$ given by $\o \vec e_3\wedge x+b$. But this is against Theorem \ref{thm.Gg}, according to which $(\o,b)=(0,0)$.

\vskip10pt\noindent{\bf Step 3: Conclusion.} By Step 2, $(\o_n,b_n)_{n\in\N}$ is bounded in $\R\times \R^2$, so that, up to a subsequence, $(\om_n,b_n)\to (\om,b)$. Relations \eqref{eq:unuEbis} and \eqref{eq:onbnE} of Step 1 yield the conclusion.
\end{proof}

\subsection{Shape optimization}
\label{subsec:shape}
We are now in a position to derive our  shape optimization result. 
\par
Let $D\subset \R^2$ be an open bounded set with Lipschitz boundary, and  fix $D'\subset\subset D$ open. 
Given $0<m<{\mathcal L}^2(D')$,  consider the family of obstacles
\beq\label{eq.fo}
{\mathcal E}_m:=\{E\subset  \overline{D'}\,:\, \text{$E$ is compact connected with }{\mathcal L}^2(E)\ge m\}.
\eeq
Let $f:\R\times \R^2\to \bar \R$. We are interested in minimizing or maximizing
the map
$$
E\in {\mathcal E_m} \mapsto f(\o_E,b_E),
$$
where $(\o_E,b_E)$ is the velocity field on $E$ of the self-equilibrated Stokes flow under the boundary condition
$\psi\in H^1(D;\R^2)$ with $\hdiv \psi=0$ according to Theorem \ref{thm.Gg}.

The following result holds

\begin{theorem}
\label{thm:shape}
If $f:\R\times \R^2\to \bar \R$ is lower semicontinuous, then the minimum problem 
$$
\min_{E\in {\mathcal E}_m} f(\o_E,b_E)
$$
is well posed. If $f$ is upper semicontinuous, then the maximum problem
$$
\max_{E\in {\mathcal E}_m} f(\o_E,b_E)
$$
is well posed.
\end{theorem}

\begin{proof}
We investigate  the minimization problem, the other one  being similar.
Consider a minimizing sequence $E_n\in {\mathcal E}_m$.
\par
Appealing to Blaschke's compactness properties of Hausdorff convergence (see \cite[Theorem 4.4.15]{AT} or  \cite[Theorem 3.16]{Falconer86}) we can always assume that 
$$E_n \mbox{ converges to  a compact  set $E\subset \bar D'$ for the Hausdorff convergence.}
$$
Further, since $E_n$ is a sequence of connected sets, $E$ is also connected.
\par
Moreover, since, for every compact set $K\subset D\setminus E$, $K\subset D\setminus E_{n}$ for $n$ large enough, we easily deduce
$$
{\mathcal L}^2(D\setminus E)\le \liminf_{n\to +\infty}{\mathcal L}^2(D\setminus E_{n})
$$
so that 
$$
m\le \limsup_{n\to +\infty}{\mathcal L}^2(E_{n})\le {\mathcal L}^2(E).
$$

We conclude that $E\in {\mathcal E}_m$ and that $ {\mathcal E}_m$ is compact in the Hausdorff topology. In particular,  ${\rm cap_2}(E)>0$ .

\par
In view of the stability result Theorem \ref{thm:stability}, $\o_{E_n} \to \o_E$ and $b_{E_n}\to b_E$. Then the lower semicontinuity of $f$  implies
$$
f(\o_E,b_E) \le \liminf_{n\to +\infty}f(\o_{E_n},b_{E_n}) 
$$
so that $E$ is a minimizer for the problem.
\EEE
\end{proof}

\begin{remark}
\label{rem:constraint}
In the spirit of the proof of Theorem \ref{thm:shape},  the shape optimization is also  well posed  in 
\beq\label{eq.domain-bis}
{\mathcal F}:=\{E\subset  \overline{D'}\,:\, \text{$E$ is compact connected with $\ell_E\subseteq E$ }\},
\eeq
where $\ell_E$ is a segment of given length $l>0$. 
\par
Indeed, if $E_n\in  {\mathcal F}$ with $E_n$  converging  to $E$ in the Hausdorff topology, then, as above $E\subseteq \overline{D'}$ and $E$ is connected. Up to a further subsequence, we can further assume that the segments $\ell_{E_n}$ converge in the Hausdorff topology to a segment $\ell_E$ of length $l$ such that $\ell_E\subseteq E$, which shows that $E\in {\mathcal F}$. The rest of the proof is identical.
\hfill\P\end{remark}

\section{A generalized  setting -- The Navier-Stokes case} \label{sec.galdiNS}

In this section where the dimension is kept to $N=2,3$ for simplicity of exposition, we propose to  tackle the quasi-static problem in the generalized setting of Subsection \ref{sub.Gen}.
To that effect we must find in this new framework an equivalent to Remark \ref{rk.switch} which will allow us to transfer 
the conditions \eqref{eq.qs-solid} to the smooth boundary $\p D$ of the fluid domain.
\par
 We only consider  the case $N=3$ and emphasize the modifications required in  the case $N=2$ as needed. 
To such an end, we introduce the rank-one tensor $\rho_E(\o\wedge(x-x_*))\otimes(\o\wedge(x-x_*))$. It is easily checked, using that $\dive(\o\wedge(x-x_*))=0$, that  the following algebraic properties hold true:
$$
\begin{cases}
\dive [(\o\wedge(x-x_*))\otimes(\o\wedge(x-x_*))]= \o\wedge(\o\wedge(x-x_*)) \\[3mm]
\dive [((x-x_*)\wedge(\o\wedge(x-x_*)))\otimes(\o\wedge(x-x_*))]=(x-x_*)\wedge[ \o\wedge(\o\wedge(x-x_*))].
\end{cases}
$$
Thus, recalling \eqref{eq.qs-solid}, we apply  the divergence theorem  to the second  and third term of the first equation as well as to those of the second equation and take into account the definition \eqref{eq.in-def} of the inertia matrix. We conclude that \eqref{eq.qs-solid}
 also reads as
\beq\label{eq.qs-solid-bis}
\left\{\begin{array}{lr}
\dis 0=\int_{\p E}\{\sigma -\rho_E(\o\wedge(x-x_*))\otimes(\o\wedge(x-x_*))+\rho_Egx_3\;\!\id\}\nu  d{\mathcal H}^2
\\[3mm]
\dis 0=\int_{\p E}(x-x_*)\wedge\{\sigma -\rho_E(\o\wedge(x-x_*))\otimes(\o\wedge(x-x_*))+\rho_Egx_3\;\!\id\}\nu  d{\mathcal H}^2\EEE.
\end{array}\right.
\eeq
Further, upon setting
$$
\kappa:=\sigma-\rho u\otimes u +\rho gx_3 \;\!\id=2\mu e(u)-  (p-\rho gx_3)\;\!\id -\rho u\otimes u,
$$
the system \eqref{eq.qs-fluid}, in which the boundary condition at $\infty$ has been replaced by \eqref{eq.gbc} on $\p D$, reads as
\beq\label{eq.qs-fluid-bis}
\left\{\begin{array}{lr}
{\rm div\;}\kappa=0,& \mbox{ on }D\setminus E\\[3mm]
\hdiv u=0,&  \mbox{ on }D\setminus E\\[3mm]
u= \psi - b,&\mbox{ on } \p D\\[3mm]
u=\o\wedge (x-x_*),&\mbox{ on }\p E.
\end{array}\right.
\eeq
We conclude that, {\it provided that the solid is neutrally buoyant, that is $\rho=\rho_E$}, then $\rho [u\otimes u]\nu=\rho_E[(\o\wedge(x-x_*))\otimes(\o\wedge(x-x_*))]\nu$ on $\p E$ and thus 
$$
\kappa\nu=\{\sigma -\rho_E(\o\wedge(x-x_*))\otimes(\o\wedge(x-x_*))+\rho_Egx_3\;\!\id\}\nu,
$$ 
so that
\eqref{eq.qs-solid-bis} merely reads as
$$
\left\{\begin{array}{lr}
\dis 0=\int_{\p E}\kappa\nu  d{\mathcal H}^2
\\[3mm]
\dis 0=\int_{\p E}(x-x_*)\wedge(\kappa\nu)  d{\mathcal H}^2.
\end{array}\right.
$$
We can thus, as in Remark \ref{rk.switch}, transfer the equilibrium conditions to $\p D$ using that ${\rm div\;}\kappa=0$ in $D\setminus E$. 
\par
 Absorbing the term 
$\rho gx_3$ into the pressure field, the problem finally reads as \eqref{eq.qs-fluid-bis} with
 \beq\label{eq.sselfeq3}
\left\{\begin{array}{lr}
 \langle\kappa\nu,1\rangle_{ \p D}=0 
\\[3mm]
 \langle (x-x_*)\wedge  \kappa\nu, 1\rangle_{ \p D}=0,
\end{array}\right.
 \eeq
where $\kappa$  is given by 
\beq\label{eq.dft}
\kappa:=\sigma-\rho u\otimes u=2\mu e(u)- p\;\!\id -\rho u\otimes u.
\eeq
  Similar arguments can be used to deal with the case $N=2$ (the weight can then be neglected), leading again to problem \eqref{eq.qs-fluid-bis}, \eqref{eq.sselfeq3}. 
 \par
  The introduction of the stress $\kappa$ is  key in addressing   a general obstacle $E$ without regular boundary.  
 \par
 
\vskip10pt
By application of  Bogovski's theorem  (see \cite[Theorem 1]{Bogovski79}) as in Remark \ref{rem:bound-uE}, we are always at liberty to assume that the boundary condition $\psi$  satisfies
\beq
\label{eq:supt-psi}
{\rm supp}\psi\subset \bar D\setminus D' \mbox{ for some Lipschitz open set $D'$ with $E\subset\subset D'\subset\subset D$}.
\eeq
This we will do from now onward. 
\par
 Assume \eqref{eq.domain} as before 
and consider the  (slightly modified)  set defined for any pair $(\o,b)\in \R^3\times \R^N$ as
\begin{multline*}
\Hspace{E}{\o}{b}(D) :=\text{Closure}_{H^1(D)}\{\phi \in C^\infty(\bar D;\R^N): \hdiv\phi=0,\\
 \phi \equiv \o \wedge (x-x_*)\mbox{ in a neighborhood of }E, \phi=-b \text{ on $\partial D$}\}.
\end{multline*}
(In the case $N=2$, $\o$ is of the form $\o\vec{e}_3$ so that the wedge product has to be understood as the wedge product of $\o$ with the $3$-vector $x+0\vec{e}_3$).

Self-equilibrated flows for a boundary datum $\psi\in H^1(D)$ with $\dive{\psi}=0$  and satisfying \eqref{eq:supt-psi} are solutions  $(u,p)$  of
 \begin{equation}
\label{eq:Navier-self}
\begin{cases}
\rho \dive(u\otimes u)-\mu\triangle u+\nabla p=0 \mbox{ in }D\setminus E\\[3mm]
 u-\psi \in \Hspace{E}{\o}{b}(D),  p\in   L^2_{loc}(D\setminus E)
\end{cases}
\end{equation}
such that
\begin{equation}
\label{eq:kappa}
\pscal{\kappa \nu}{1}_{\p D}=0\qquad\text{and}\qquad \pscal{(x-x_*)\wedge \kappa \nu}{1}_{\p D}=0,
\end{equation} 
where $\kappa$ is given in \eqref{eq.dft}.
\par
The search for self-equilibrated solutions \eqref{eq:Navier-self}-\eqref{eq:kappa} can be reduced to the following problem involving the velocity field alone: find a pair $(\o,b)$ and 
 $u$ with $u-\psi\in \Hspace{E}{\o}{b}(D)$ such that
\begin{equation}
\label{eq:weak-navier}
 2 \mu \int_{ D} e(u)\cdot e(\varphi)\,dx=\rho\int_{ D} (u\otimes u)\cdot e(\varphi)\,dx
\end{equation}
for every $\varphi \in \Hspace{E}{\tilde \o}{\tilde b}(D)$ {\it and} every
$({\tilde \o},{\tilde b})\in \R^3\times\R^N$.  Note that, by definition of $\Hspace{E}{\o}{b}(D)$, $u$ is in particular in $H^1(D)$.
\par
Indeed the following lemma holds.

\begin{lemma}
\label{lem:equiv}
If $u$ with $u-\psi\in \Hspace{E}{\o}{b}(D)$ solves \eqref{eq:weak-navier} then there exists $p\in L^2_{loc}(D\setminus E)$ such that $(u,p)$ solves \eqref{eq:Navier-self} and \eqref{eq:kappa}. Conversely, if  $(u,p)$ solves \eqref{eq:Navier-self} and \eqref{eq:kappa}, then $u$ solves \eqref{eq:weak-navier}.
\end{lemma}

\begin{proof}
For a given pair $(\tilde  \o, \tilde b)$, equality \eqref{eq:weak-navier} is satisfied if and only if it is also satisfied for every  $\tilde \varphi \in H^1(D)$ with $\tilde \varphi\equiv 0$ in a neighborhood of $E$ and $\tilde \varphi=-(\tilde \o\wedge (x-x^*)+\tilde b)$ on $\partial D$. 

In particular, choosing  smooth  tests with compact support in $D\setminus E$  (so taking $\tilde\o=\tilde b=0$ for those)  we recover a pressure field $p\in L^2_{loc}(D\setminus E)$ such that the tensor $\kappa$ in \eqref{eq.dft} satisfies
$$
\dive \kappa=0 \quad\text{on }D\setminus E.
$$
Moreover, proceeding as in  Section \ref{sub.Gen},  $p\in L^2(D\setminus \overline{D'})$  (recall \eqref{eq:supt-psi}),  hence $\kappa\in L^2(D\setminus \overline{D'})$, so that integration by parts entails that
$$
\tilde b \cdot \langle \kappa \nu,  1\rangle_{\partial D}+\tilde \o \cdot \langle (x-x_*)\wedge \kappa \nu,  1 \rangle_{\partial D}=0.
$$
Relations \eqref{eq.sselfeq3} follow since $\tilde \o$ and $\tilde b$ are arbitrary. 
\par
The reverse implication is similar, and the proof is concluded.
\end{proof}

\EEE

\subsection{Existence of self-equilibrated solutions}\label{subsec:ses-navier}
  
We prove the existence of self-equilibrated solutions by addressing the equivalent problem \ref{eq:weak-navier} through a Banach fixed point argument, provided that the external boundary condition satisfies a suitable smallness assumption  -- essentially akin to imposing a small enough Reynolds number -- and that the mass density of the fluid is also that of the obstacle (neutral buoyancy).

 The proof will use the following variant of Korn's inequality.

\begin{lemma}[Korn]
\label{lem:Korn} 
Let $E$ satisfy \eqref{eq.domain}. There exists a constant $C_{D,E}>0$ such that, for every $u\in H^1(D;\R^N)$ with $u=0$ in a neighborhood of $E$, 
\begin{equation}
\label{eq:kornE}
\|u\|_{H^1(D)}\le C_{D,E}\|e(u)\|_{L^2(D)}.
\end{equation}
Further, if $E$ is restricted to be such that ${\mathcal L}^N(E)\ge m$ for some fixed constant $m>0$,  then $C_{D,E}= C_{D,m}$, that is the constant in \eqref{eq:kornE} depends
only on $D$ (and $m$ which will remain fixed).
\end{lemma} 
\begin{proof}
By Korn's inequality on $D$,
\begin{equation}
\label{eq:equiv}
\|u\|_{H^1(D)}\le C_D(\|e(u)\|_{L^2(D)}+\|u\|_{L^2(D)})
\end{equation}
so that it suffices to prove that, for some constant $C_{D,E}$,
$$
\|u\|_{L^2(D)}\le C_{D,E}\|e(u)\|_{L^2(D)}.
$$
We proceed by contradiction. Assume that $u_n\in H^1(D)$ is such that
$$
\|u_n\|_{L^2(D)}>n\|e(u_n)\|_{L^2(D)}.
$$
 We may assume  that $u_n$ is smooth   with, upon renormalizing,
$$
\|u_n\|_{L^2(D)}=1\qquad\text{and}\qquad \|e(u_n)\|_{L^2(D)}<\frac{1}{n}.
$$

By \eqref{eq:equiv}, $(u_n)_{n\in\N}$ is bounded in $H^1(D)$. By  compact Sobolev embedding, for a subsequence, 
$
u_n\to u$ strongly in $L^2(D)
$
for some $u\in H^1(D)$ such that
$
\|u\|_{L^2(D)}=1$ and $e(u)=0 \text{ in $D$}$. Thus $u=\o\wedge x+b$ for some pair $(\o,b)\in \R^3\times\R^N$.

Further, using \eqref{eq:equiv} again,
$
u_n\to u \mbox{ strongly in }H^1(D),
$
 so that  \cite[Theorem 5.11]{KLV} implies that, for a subsequence, 
$u_n\to \tilde u$  ${\rm cap}_2$-q.e. on $D$. In particular, since $u_n=0$ on $E$, $\tilde u=0$ ${\rm cap}_2$-q.e. on $E$.  Exactly as in the proof of Proposition \ref{prop:galdi0}, this yields in turn that $\o=b=0$, in contradiction with $\|u\|_{L^2(D)}=1$.

To prove the last part of the Lemma we argue by contradiction as well and, as before, consider   a sequence  of obstacles $E_n$
and a sequence $u_n\in  H^1(D)$,  with
$u_n=0$ in a neighborhood of $E_n$, such that
$$
\|u_n\|_{L^2(D)}> n\|e(u_n)\|_{L^2(D)} \text{ and } \|u_n\|_{L^2(D)}=1.
$$
But then   $u_n \to u$ strongly in $H^1(D;\R^N)$ with $u=\o\wedge x+b$ for some pair $(\o,b)\ne (0,0)$. Now, since
$$
\int_{D} \limsup_n 1_{E_n}\,dx \ge \limsup_n \int_D 1_{E_n}\,dx \ge m
$$
we infer that the set-theoretic limsup set
$
F:=\{x\in D\,:\, \limsup_n 1_{E_n}(x)=1\}
$
of the sequence $(E_n)_{n\in\N}$
satisfies ${\mathcal L}^N(F)\ge m$. Since, up to a subsequence, $u_n\to u$  a.e. in $D$, we obtain that, for a.e. $x\in F$,
$$
\o \wedge x+b=0
$$
which is against $(\o,b)\not=(0,0)$.
\end{proof}

\begin{theorem}[\bf A generalized quasi-static theorem] 
\label{thm.qs} 
Assume that  $E\subset\subset D'\subset\subset D$  satisfies \eqref{eq.domain}, that the dimension $N$ is $2$ or  $3$  and that $\rho=\rho_E$. \EEE Consider $\psi\in H^1(D;\R^N)$ with ${\rm div\;} \psi=0$ and with support in $\bar D\setminus D'$.
\par
Then, for some $M_{D,E}>0$ and $M'_{D,E}>0$ only depending on $D$ and $E$, if
$$
({\rho}/{\mu})\|\psi\|_{H^1(D)}\le M_{D,E}, 
$$
there exists a unique pair $(\o_E, b_E)\in \R^{3}\times\R^N$ and a unique  $u_E$ with $u_E-\psi\in \Hspace{E}{\o_E}{b_E}(D)$  and 
$\|e(u_E)\|_{L^2(D)}\le M'_{D,E}$
such that \eqref{eq:weak-navier} is satisfied.

\end{theorem}

\begin{proof}
We  will prove the existence of a solution to \eqref{eq:weak-navier} through a Banach fixed point argument.

\vskip10pt\noindent{\bf Step 1.}
Set
$$
X_{E}:=\bigcup_{\o,b} \Hspace{E}{\o}{b}(D)
$$
We claim that $X_E$ is a closed subspace of $H^1(D)$.
\par
In order to prove this, let us first notice that there exists $C_{D,E}>0$ such that for every $v\in \Hspace{E}{\o}{b}(D)$ we have
\begin{equation}
\label{eq:est-obv}
|\o|+|b|\le C_{D,E} \|e(v)\|_{L^2(D)}.
\end{equation}
 This is a consequence of the continuity of the trace operator from $H^1(D)$ to $L^2(\partial D)$ and of the Korn inequality of Lemma \ref{lem:Korn} applied to the function $v-\o\wedge (x-x_*)$, together with the fact that infinitesimal rigid body motions form a finite dimensional space.
\par
 To prove that $X_E$ is a closed subspace of $H^1(D)$, let $v_n\in X_{E}$ with 
$v_n \in \Hspace{E}{\o_n}{b_n}(D)$ and $v_n\to v$ strongly in $H^1(D)$. By definition of the space $\Hspace{E}{\o_n}{b_n}(D)$, it is not restrictive to assume that $v_n$ is smooth and $v_n=\o_n\wedge (x-x_*)$ in a neighborhood of $E$. 
 From \eqref{eq:est-obv} we get
$$
|\o_n|+|b_n|\le  C_{D,E} \|e(v_n)\|_{L^2(D)}.
$$
Thus $(\o_n,b_n)$ is a bounded sequence in $\R^3\times \R^N$, so that, up to a subsequence, we may assume
$\o_n\to \o$ and $b_n\to b$. 
\par
 By application of  Bogovski's theorem (see \cite[Theorem 1]{Bogovski79}) we can construct a smooth function $\eta_n\in C^\infty(\bar D;\R^N)$ such that $\eta_n\to 0$ strongly in $H^1(D)$ and
$$
\dive{\eta_n}=0 \text{ in $D$},\quad \eta_n=-\o_n\wedge (x-x_*)+\o\wedge (x-x_*) \text{ in $D'$}\quad \eta_n=-b_n+b \text{ on $\partial D$}.
$$
Then the smooth function $\tilde v_n:=v_n+\eta_n$  belongs to  $\Hspace{E}{\o}{b}(D)$
with $\tilde v_n\to v$ strongly in $H^1(D)$. 
This shows that $v\in \Hspace{E}{\o}{b}(D)$, and thus $v\in X_E$.

\vskip10pt\noindent{\bf Step 2.} 

For every $v$ with $v-\psi\in X_{E}$, let  $T(v)$ with $T(v)-\psi\in X_{E}$ 
\EEE
 be the solution to
\beq\label{eq.Tveq}
2\mu\int_{D}  e(T(v))\cdot e(\ph) \ dx= \rho \int_{D} (v\otimes v)\cdot e(\ph) \ dx
\eeq
for every $\varphi\in  X_E$. Since  $v\otimes v\in L^4(D)$ by Sobolev embedding, the solution $T(v)$ can be recovered by minimizing the functional
\begin{equation}
\label{eq:funct}
u \mapsto \int_{D}|e(u)|^2\,dx-\frac{\rho}{\mu}\int_{D}(v\otimes v):e(u)\,dx
\end{equation}
 over the functions  $u$ such that $u-\psi\in X_{E}$. The minimization is well posed because $X_E$ is closed in $H^1(D)$ by Step 1 and because a control on the $L^2$ norm of the symmetrized gradient is enough to get a bound in  $H^1$-norm. Indeed, thanks to 
 Korn's inequality on $D$, 
\begin{equation}
\label{eq:Korn-v}
\|u\|_{H^1(D)}\le  C_D \left( \|u\|_{L^2(\partial D)}+\|e(u)\|_{L^2(D)}\right).
\end{equation}
 In view of estimate \eqref{eq:est-obv},  
$$
\|u\|_{L^2(\p D)}\le  C_D (|\o|+|b|+\|\psi\|_{L^2(\p D)})
\le C_{D,E}\left( \|e(u)\|_{L^2(D)}+\|\psi\|_{H^1(D)}\right).
$$
Hence
\beq
\label{eq:Korn-v-bis}
\|u\|_{H^1(D)}\le C_{D,E}\left( \|e(u)\|_{L^2(D)}+\|\psi\|_{H^1(D)}\right).
\eeq
To get a bound for the solution $T(v)$, we first note that the boundary datum $\psi$, because of where it is supported, is an admissible competitor for the minimization of \eqref{eq:funct}. Thus,  
\begin{multline*}
\|e(T(v))\|^2_{L^2(D)}
\le \frac{\rho}{\mu}\int_{D}(v\otimes v)\cdot e(T(v))\,dx+\int_{ D}|e(\psi)|^2\,dx-\frac{\rho}{\mu}\int_{D}(v\otimes v)\cdot e(\psi)\,dx\\
\le \frac{1}{2}\|e(T(v))\|^2_{L^2(D)}+\frac{\rho^2}{2\mu^2}\|v\|^4_{L^4(D)}+\|e(\psi)\|^2_{L^2(D)}+\frac{\rho^2}{2\mu^2}\|v\|^4_{L^4(D)}+\frac{1}{2}\|e(\psi)\|^2_{L^2(D)}
\end{multline*}
which yields
$$
\|e(T(v))\|_{L^2(D)}\le \sqrt{3}\|e(\psi)\|_{L^2(D)}+ \sqrt 2 \frac{\rho}{\mu} \|v\|^2_{L^4(D)}.
$$
Thanks to \eqref{eq:Korn-v-bis} and  Sobolev embedding, 
\begin{equation}
\label{eq:L4}
\|v\|_{L^4(D)}\le C_D\|v\|_{H^1(D)}\le  C_{D,E}\left( \|e(v)\|_{L^2(D)}+\|\psi\|_{H^1(D)}\right),
\end{equation}
so that we conclude to the existence of a constant  $ C_{D,E}>0$ such that 
\begin{equation}
\label{eq:estTv}
\|e(T(v))\|_{L^2(D)}\le  C_{D,E} \left(\|\psi\|_{H^1(D)}+\frac{\rho}{\mu}\|\psi\|^2_{H^1(D)}+\frac{\rho}{\mu} \|e(v)\|^2_{L^2(D)}\right).
\end{equation}

\vskip10pt\noindent{\bf Step 3.}  We claim that we can find $R>0$ such that, considering the closed subset of $H^1(D)$
$$
K:=\{v:\, v-\psi\in X_{E}, \|e(v)\|_{L^2(D)}\le R\},
$$
the map $v\mapsto T(v)$ defined in Step 2 is a strict contraction from $K$ into itself with respect to the distance 
$$
d(v,v'):=\|e(v)-e(v')\|_{L^2(D)}.
$$
 This distance is equivalent to that induced by the norm of $H^1(D)$ because, by \eqref{eq:Korn-v} then  \eqref{eq:est-obv} applied to $v-v'\in X_E$,
$$
\|v-v'\|_{H^1(D)}\le C_D(\|v-v'\|_{L^2(\partial D)}+\|e(v)-e(v')\|_{L^2(D)})
\le C_{D,E}\|e(v)-e(v')\|_{L^2(D)}.
$$
In view of \eqref{eq:estTv}, if 
\begin{equation}
\label{eq:psi1}
C_{D,E}\left(\|\psi\|_{H^1(D)}+\frac{\rho}{\mu}\|\psi\|^2_{H^1(D)}\right) \le \alpha,
\end{equation}
then,  for $\|e(v)\|_{L^2(D)}\le R$,  
$$
\|e(T(v))\|_{L^2(D)}\le C_{D,E} \frac{\rho}{\mu}R^2+\alpha.
$$
If $0<\beta<1$ and $R>0$ are such that
\begin{equation}
\label{eq:R}
C_{D,E}\frac{\rho}{\mu}R^2+\alpha= \beta R,
\end{equation}
 we get $\|e(T(v))\|_{L^2(D)}\le \beta R<R$, so that $T(v)\in K$. Making the quadratic expression associated to \eqref{eq:R} a perfect square, possible values for $R$ and $\alpha$ are 
\begin{equation}
\label{eq:R2}
R:=\frac{\beta}{2 C_{D,E}\frac{\rho}{\mu}}\qquad\text{and}\qquad \alpha:=\frac{\beta^2}{4 C_{D,E}\frac{\rho}{\mu}}.
\end{equation}
The restriction \eqref{eq:psi1} on $\psi$ now reads as
\begin{equation}
\label{eq:psi}
\frac{\rho}{\mu}\|\psi\|_{H^1(D)}+\frac{\rho^2}{\mu^2}\|\psi\|^2_{H^1(D)} \le \frac{\beta^2}{4  C^2_{D,E}}.
\end{equation}
Let us now check that $T$ is a contraction on $K$. If $v,v'\in K$, then from \eqref{eq.Tveq} 
$$
 d(T(v),T(v')) =\|e(T(v))-e(T(v'))\|_{L^2(D)}
\le \frac{\rho}{ 2\mu} \|v\otimes v-v'\otimes v'\|_{L^2(D)}.
$$
Then, using a simple estimate as in \cite[Equation (16.12)]{tartar} together with  \eqref{eq:L4} , this also reads as 
\begin{multline*}
 d(T(v),T(v')) \le \frac\rho\mu (\|v\|_{L^4(D)}  + \|v'\|_{L^4(D)})
\|v-v'\|_{L^4(D)}\\[3mm]\le C^2_{D,E} \frac\rho\mu (\|e(v)\|_{L^2(D)}  + \|e(v')\|_{L^2(D)}+2\|\psi\|_{H^1(D)})\|e(v-v')\|_{L^2(D)}\\[3mm]\le  2  C^2_{D,E}\frac\rho\mu (R+\|\psi\|_{H^1(D)}) d(v,v').
\end{multline*}
Recalling \eqref{eq:R2} and \eqref{eq:psi} we get
$$
d(T(v),T(v')) \le \left(C_{D,E}\beta+\frac{\beta^2}{2}\right)d(v,v'),
$$
that is $T$ is a strict contraction provided that $\beta$ is chosen small enough.

\vskip10pt\noindent{\bf Step 4: Conclusion} In view of Step 2 and Step 3, the conclusion follows by applying the Banach Fixed Point Theorem and choosing
\beq\label{eq.val-M}
M_{D,E}:=\frac{-1+\sqrt{1+{\beta^2}{ / C^2_{D,E} }}}{2}\qquad\text{and}\qquad M'_{D,E}:=\frac{\beta}{2 C_{D,E}\frac{\rho}{\mu}}.
\eeq
with  $0<\beta<1$ small enough.
\end{proof}
\EEE

\begin{remark}
\label{rem:uniform}
 If the obstacle $E$ is further constrained to satisfy ${\mathcal L}^N(E)\ge m$ for some fixed $m>0$, the constants $M_{D,E}$ and $M'_{D,E}$ in Theorem \ref{thm.qs}  only depend on $D$ and $m$ (and are thus renamed $M_{D,m}$ and $M'_{D,m}$).
 In view of \eqref{eq.val-M}, this amounts to proving that  the constant $C_{D,E}$ appearing in \eqref{eq:est-obv}  is uniformly bounded for $E$ varying in $\bar D'\subset D$ and provided that  $v\in \Hspace{E}{\o}{b}(D)$. Indeed this is the only time in the proof of Theorem \ref{thm.qs} that the constants used in the various estimates depend a priori on both $D$ and $E$. But this is an immediate consequence of the latter part of Lemma \ref{lem:Korn}. 
\par
As a consequence, thanks to inequality \eqref{eq:Korn-v-bis}, the solution $u_E$ satisfies

$$
\|u_E\|_{H^1(D)}\le
C_{D,m} (\|e(u_E)\|_{L^2(D)}+\|\psi\|_{H^1(D)})
\le C_{D,m} \left(M'_{D,m}+\frac\mu\rho M_{D,m}\right)=: C_{D,m}.
$$
\hfill\P\end{remark}

\subsection{An optimization result in dimension two}
\label{subsec:stab-navier}

In this subsection we extend the shape optimization result of Theorem \ref{thm:shape} to the 2d Navier-Stokes setting.
\par
 As in  Subsection \ref{subsec:shape},  let $D\subset \R^2$ be an open bounded set with Lipschitz boundary, and  fix $D'\subset\subset D$ open. 
Given $0<m<{\mathcal L}^2(D')$, consider the  same family of obstacles $
{\mathcal E}_m$ as in \eqref{eq.fo}. 
\par
Consider $\psi\in H^1(D;\R^2)$ with $\hdiv \psi=0$ and with support in $D\setminus D'$. By Theorem \ref{thm.qs}  and Remark \ref{rem:uniform}, we know that, if $({\rho}/{\mu})\|\psi\|_{H^1(D)}$ is sufficiently small,  there exists for every $E\in {\mathcal E}_m$ a unique self-equilibrated flow $u_E$ with associated velocity $(\o_E,b_E)\in \R\times \R^2$ on $E$.
\par
The following result holds.

\begin{theorem}
\label{thm:shape-navier}
Let $\psi\in H^1(D;\R^2)$ with ${\rm div}\;\!\psi=0$ and with support in $\bar D\setminus D'$ such that
$$
\frac{\rho}{\mu}\|\psi\|_{H^1(D;\R^2)}\le  M_{D,m},
$$
where  $M_{D,m}$  is given by Theorem \ref{thm.qs} and Remark \ref{rem:uniform}. 
\par
If $f:\R\times \R^2\to \bar \R$ is lower semicontinuous, then the minimum problem 
$$
\min_{E\in {\mathcal E}_m} f(\o_E,b_E)
$$
is well posed,  while if $f$ is upper semicontinuous, then the maximum problem
$$
\max_{E\in {\mathcal E}_m} f(\o_E,b_E)
$$
is well posed.
\end{theorem}

\begin{proof}
The proof of the optimization result of Theorem \ref{thm:shape} for the Stokes case can be adapted to the present setting provided that we establish the stability of self-equilibrated solutions under convergence of the obstacles.
\par
Assume that $E_n,E\in \mathcal{E}_m$ with
$$
E_n\to E\qquad\text{in the Hausdorff topology.}
$$
For notational simplicity, let us set $u_{E_n} :=u_n$, $\o_{E_n}:=\o_n$, $b_{E_n}:=b_n$.
We need to check that
\begin{equation}
\label{eq:weakunN}
u_{n}\weak u_E\qquad\text{weakly in }H^1(D;\R^2)
\end{equation}
and
\begin{equation}
\label{eq:rigidN}
\o_{n}\to \o_E \qquad\text{and}\qquad b_{n}\to b_E.
\end{equation}
By Remark \ref{rem:uniform} $(u_n)_{n\in\N}$ is bounded in $H^1(D)$ so that, up to a subsequence
$$
u_n\weak u\qquad\text{weakly in }H^1(D).
$$
Thus $\|e(u)\|_{L^2(D)}\le  M'_{D,m}$, where $ M'_{D,m}$ is the constant given by Theorem \ref{thm.qs}  and Remark \ref{rem:uniform}. 
\par
By Lemma \ref{lem:Korn}, $u_n-\psi+\o_n\wedge (x-x_*)$ is bounded in $H^1(D)$ as well, so that,
taking the trace on $\partial D$, we infer that, up to a subsequence,
\begin{equation}
\label{eq:onbn}
\o_n\to \o\qquad\text{and}\qquad b_n\to b.
\end{equation}
Following the arguments of Step 1 in the proof of Theorem \ref{thm:stability}, we deduce that
\ $u-\psi\in \Hspace{E}{\o}{b}(D).$

We finally prove that \eqref{eq:weak-navier} holds true for $u$.
This shows that $u=u_E$, and consequently $\o=\o_E$ and $b=b_E$; \eqref{eq:weakunN} and \eqref{eq:rigidN} follow.
\par
It is not restrictive to assume that $\varphi=\tilde \o \vec e_3\wedge (x-x_*)$ on a neighborhood $U$ of $E$, so that,  since $E_n\subset U$ for $n$ large enough thanks to Hausdorff convergence, 
$$
 \varphi \in \Hspace{E_n}{\tilde \o}{\tilde b}(D). 
$$
 In view of the compact embedding of $H^1(D)$ into $L^4(D)$,  
$$
u_n\otimes u_n \to u\otimes u\qquad\text{strongly in }L^2(D).
$$
 and we obtain
\begin{multline*}
2\mu \int_{D} e(u) \cdot e(\varphi)\,dx-\rho\int_{D} (u\otimes u) \cdot e(\varphi)\,dx\\
=\lim_{n\to +\infty}\left( 2\mu \int_{D} e(u_n) \cdot e(\varphi)\,dx-\rho\int_{D} (u_n\otimes u_n) \cdot e(\varphi)\,dx
\right)=0
\end{multline*}
Thus \eqref{eq:weak-navier} holds true.

\end{proof}

\begin{remark} In contrast with the mobility setting of Subsection \ref{subsec:stability}, we cannot in the present quasi-static setting conclude to the existence of optimizers as in Theorem \ref{thm:shape-navier} under the conditions detailed in Remark \ref{rem:constraint} because   the proof of the convergences \eqref{eq:onbn} is  conditional upon the volume constraint $\mathcal{L}^2(E_n) \ge m$. The proof of  the corresponding property \EEE in the stability  result Theorem \ref{thm:stability}, that is Step 2 of that proof, uses linearity, a feature no longer  available in the present setting. Also,  our proof of  the $E$-independence of the $H^1$-bound on the solution to the quasi-static problem  when that obstacle varies  is also based on the volume constraint.  
\hfill\P\end{remark}

\begin{remark}\label{rem.wzm}
  In the introductory remarks, we alluded to a possible solving of the quasi-static problem of self-equilibration in a (non inertial) reference frame attached to the obstacle.  Indeed, one could perform an analysis of the resulting model along similar lines, the main difference being in the details of the proof of Theorem \ref{thm.qs}.  Under assumption \eqref{eq.fo}, a theorem identical to Theorem \ref{thm:shape-navier} would be produced.
\hfill\P\end{remark}\EEE

\section{A few explicit examples}\label{sec.example}

In this final section we propose to explicitly compute $\o_E,b_E$ for the mobility problem when the obstacle $E$ has three specific shapes and when dealing with a 2d pure shear flow, that is when 
$$
\psi (x) :=Ax \text{ on $\p D$}, \qquad  A=\begin{pmatrix}0&1\\0&0\end{pmatrix}, \qquad x:=(x_1,x_2).
$$

\medskip

Our first computation is when the domain $D$ is a large open disk of radius $R>1$ centered at $0$ and the obstacle $E$ is a horizontal needle (the line segment $\ell:=[-1,1]\times\{0\}$). Then the solution to the Stokes system with $u=0$ on $x_2=0$ is $u=(x_2,0), p=0$. But then 
$$
\sigma= 2\mu e(u)-p\;\!\id=\mu  \begin{pmatrix}0&1\\1&0\end{pmatrix}
$$ 
and thus $\sigma\nu= \mu/{R}\; (x_2,x_1)$ on $\p D$. Consequently,
$$
\int_{\p D}\sigma\nu\ d\hh=\int_{\p D}x\wedge\sigma\nu\ d\hh=0
$$
so that we obtain that $\o_{ E }=b_{ E }=0$. Since, according to Remark \ref{rem:constraint}, a horizontal needle is a valid candidate for velocity minimization,  $\ell$ does indeed provide a minimizer for any  non-negative function  $f$ of $(\o,b)$ with $f(0,0)=0$. Note that we might as well have chosen any countable union of line segments of the form $[\alpha_i,\alpha_{i+1}]\times \{0\}, i=1,....$ with $\alpha_i$ strictly increasing and $-1<\alpha_1<\lim_i\alpha_i<1$. The result would be identical. 

Our second computation is the same setting, but for the fact that the needle $E$ is vertical (the line segment $\ell:=\{0\}\times[-1,1]$). Then the solution to the Stokes system with $u=(x_2,0)$ on $x_1=0$ is $u=(x_2,0), p=0$. But then  we are back to the computations of the first example. In this case we get that $\o_{ E }=-1, b_{ E }=0$.  Since, according to Remark \ref{rem:constraint}, a vertical needle is also a valid candidate for velocity minimization,  $\ell$ does indeed provide a possible candidate for the maximization of $|\omega|$. 

These elementary explicit solutions are for obstacles without volume and are hence physically unrealistic. Whether $|\o|$ can attain or exceed the value $1$ among smooth head-tail symmetric obstacles in a 2d pure shear flow at infinity is presently undecided, although such shapes have been hypothesized on unbounded domains \cite{Bretherton1962,SKS2013}. In that setting, if there exists a geometry such that $\o = 0$, then a pure rotation of that geometry yields another obstacle with $|\o| \geq 1$.

Our third computation is that of an obstacle $E=\bar B'$ where $B'$ is a small disk  of radius $R'\le 1$ centered at $0$ with $D$ still a larger disk as in the previous example. We rephrase the problem in polar coordinates and seek a solution to the Stokes problem with 
$$
\begin{cases}
u= R'\o(0,1)&\mbox{ on } \p B'\\[3mm]
u=R(\sin\theta\cos\theta,-\sin^2\theta) &\mbox{ on } \p B
\end{cases}
$$
in the basis $\vec{e}_r,\vec{e}_\theta$. We seek a solution to the Stokes system of the form 
$$
u= {\rm Re}\EEE\left(\sum_{k=0}^\infty \huk(r) e^{ik\theta}\right), \qquad  p= {\rm Re}\EEE\left(\sum_{k=0}^\infty \hat p_k(r) e^{ik\theta}\right)
$$
with $\huk(r):=(\hur{k}(r),\hut{k}(r))$.
Using the orthogonality of the Fourier modes and the expression for the Cauchy stress in polar coordinates
leads to the following expressions for \eqref{eq.selfequi}:
\beq\label{eq.sefpol}
F:=\int_{\p E}\sigma\nu d\hh=-\pi\left(\begin{array}{c} \dis  {\rm Re} \EEE[-R'\hat p_1+2R'\mu \hur{1}']- \mu \EEE  {\rm Im} \EEE[i\hur{1}+R'\hut{1}'-\hut{1}]\\
[3mm]\dis  {\rm Im} \EEE[-R'\hat p_1+2R'\mu \hur{1}']+ \mu \EEE  {\rm Re} \EEE[i\hur{1}+R'\hut{1}'-\hut{1}] \end{array}\right)\eeq

\beq\label{eq.setpol}
T:=\int_{\p E}x\wedge\sigma\nu d\hh=-\pi  \mu \EEE  {\rm Re} \EEE[R'^2\hut{0}'-R'\hut{0}]\vec{e}
\eeq
where $\vec{e}$ is perpendicular to the plane.

Now, it is easily seen that the $0$-mode $\hut{0}$ satisfies
$$
\frac1r (r\hut{0}')'-\frac{\hut{0}}{r^2}=0,
$$
and thus that the solution is of the form $\hut{0}(r)=A/r+Br$. The boundary conditions on $\hut{0}$ are
$$
\begin{cases}A/R'+BR'=R'\o\\\dis {A/R}+BR=-R\fint_0^{2\pi}\sin^2\theta d\theta.\end{cases}
$$
From this we see the associated torque $T$ in \eqref{eq.setpol} is $4\pi A$, so that it can only be $0$  if  $A=0$. In that case $B=\o$ and  the second equation in the previous $2\times2$ system yields
$\o=-1/2.$ We do not need to explicitly compute the solutions $\hur{1}, \hut{1}$ that enter the expression for the resultant force $F$ in \eqref{eq.sefpol} since we know from Theorem \ref{thm.G} that there is a (classical Galdi-type) solution for the self-equilibrated problem in the current setting.

\medskip
The three situations described above are unfortunately the only ones for which an analytical solution is readily derived. The results ($\o=0$ for a horizontal needle, $\o=-1/2$ for a disk and $\o=-1$ for a vertical needle) lead us to put forth the following conjecture for the mobility problem:

\medskip

\noindent
For all obstacles within the class defined by \eqref{eq.domain-bis} in a simple shear flow $\psi(x) = (x_2, 0)$, the angular velocity satisfies $\o \leq 0$. Moreover, the degenerate vertical ellipse realizes the maximum of $|\o|$ within this class.

\section*{Statement and declarations:}

A.G. acknowledges support from PRIN 2022 (Project no. 2022J4FYNJ), funded by MUR, Italy, and the European Union -- Next Generation EU, Mission 4 Component 1 CUP F53D23002760006.
He is also member of the Gruppo Nazionale per l’Analisi Matematica, la Probabilit\`a e le loro Applicazioni (GNAMPA) of the Istituto Nazionale di Alta Matematica (INdAM). 

\par  Data sharing not applicable to this article as no datasets were generated or analyzed during the current study.

\end{document}